\documentclass[12pt,a4paper,oneside,reqno]{article}

\usepackage[utf8]{inputenc}
\usepackage[english]{babel}
\usepackage[left=1in,right=1in,top=1in,bottom=1in]{geometry}
\usepackage{amsfonts,amssymb,amsmath,amsthm}
\usepackage[perpage,symbol]{footmisc}
\usepackage{longtable,array}
\newcolumntype{L}[1]{>{\raggedright\arraybackslash$}p{#1}<{$}}
\usepackage{enumerate}
\usepackage{cite}
\usepackage{xcolor}

\usepackage{hyperref}
\hypersetup{hidelinks}
\usepackage{fancyhdr}
\theoremstyle{plain}
\newtheorem{theorem}{Theorem}
\newtheorem{lemma}[theorem]{Lemma}
\newtheorem{proposition}[theorem]{Proposition}
\newtheorem{corollary}[theorem]{Corollary}
\theoremstyle{definition}
\newtheorem{remark}[theorem]{Remark}
\newtheorem{definition}[theorem]{Definition}

\makeatletter
\@removefromreset{theorem}{section}\@removefromreset{equation}{section}

\makeatother
\usepackage{adjustbox}
\usepackage{tikz}\usetikzlibrary{arrows}

\IfFileExists{fouriernc.sty}{\usepackage{fouriernc}}{}

\newcommand{\doi}[1]{\href{https://doi.org/#1}{doi:\detokenize{#1}}}
\newcommand{\arxiv}[1]{\href{https://arxiv.org/abs/#1}{arXiv:#1}}
\newcommand{\arxivsearch}[1]{\href{https://arxiv.org/search/?query=#1}{arXiv}}

\usepackage{stmaryrd}
\newtheorem{question}{Question}

\newcommand{\A}{\mathbf{A}}

\usepackage{url}

\begin{document}

\noindent\fbox{\parbox{\dimexpr\textwidth-2\fboxsep-2\fboxrule\relax}{\small
\textbf{Corrected version (verified by Claude).} This is a corrected version of the paper P. D. Beites, B. L. M. Ferreira, I. Kaygorodov,
\emph{Transposed Poisson structures},
Results in Mathematics, 79 (2024), no.~2, 93,
\doi{10.1007/s00025-023-02107-x}.}}

\bigskip

\begin{center}
{\Large Transposed Poisson structures}

\bigskip

{\bf Patr\'{i}cia Damas Beites, Bruno Leonardo Macedo Ferreira, \\ and Ivan Kaygorodov}
\end{center}

\medskip

\noindent{\small\textbf{Abstract.} This paper surveys known results in the theory of transposed Poisson algebras and establishes some new results on this subject. A list of open questions for future research is also given, together with notes on the questions that have been answered.}

\bigskip


\bigskip


\bigskip

\tableofcontents

\bigskip

\section*{Introduction}\addcontentsline{toc}{section}{Introduction}

Poisson algebras arose from the study of Poisson geometry in the 1970s and have appeared in an extremely wide range of areas in mathematics and physics, such as Poisson manifolds, algebraic geometry, operads, quantization theory, quantum groups, and classical and quantum mechanics. The study of Poisson algebras also led to other algebraic structures, such as 
noncommutative Poisson algebras \cite{xu94}, 
generic Poisson algebras \cite{ksu18,klms14},
Poisson bialgebras \cite{nb13,lbs20},
algebras of Jordan brackets and generalized Poisson algebras \cite{ck07,ck10,kantor92,k17, z21,msz01},
Gerstenhaber algebras \cite{ks96},
$F$-manifold algebras \cite{ls21},
Novikov-Poisson algebras \cite{xu97},
quasi-Poisson algebras \cite{billig},
double Poisson algebras \cite{v08},
Poisson $n$-Lie algebras \cite{ck16}, etc.
The study of all possible Poisson structures with a certain Lie or associative part is an important problem in the theory of Poisson algebras \cite{kk21,said2, YYZ07}.
Recently, Bai, Bai, Guo, and Wu introduced a dual notion of a Poisson algebra, the transposed Poisson algebra, by exchanging the roles of the two binary operations in the Leibniz rule defining Poisson algebras \cite{bai20}.
They showed that transposed Poisson algebras not only share common properties with Poisson algebras, such as closedness under tensor products and Koszul self-duality as an operad, but also admit a rich class of identities. More significantly, a transposed Poisson algebra naturally arises from a Novikov--Poisson algebra by taking the commutator Lie algebra of the Novikov algebra. Consequently, the classical construction of a Poisson algebra from a commutative associative algebra with a pair of commuting derivations has an analog for transposed Poisson algebras which requires only one derivation. More broadly, transposed Poisson algebras also arise when the commutator is taken in pre-Lie Poisson algebras and in two other Poisson-type algebras. 
The ${\rm Hom}$- and ${\rm BiHom}$-versions of transposed Poisson algebras are considered in \cite{hom, bihom}.
A bialgebra theory for transposed Poisson algebras is studied in \cite{lb23}.
Some new examples of transposed Poisson algebras are constructed by applying the Kantor product of multiplications on the same vector space \cite{FK21}. 
Ferreira, Kaygorodov, and Lopatkin established a relation between $\frac{1}{2}$-derivations of Lie algebras and
transposed Poisson algebras \cite{FKL}.
These ideas were used to describe all transposed Poisson structures
on the Witt algebra \cite{FKL}, the Virasoro algebra \cite{FKL},
the algebras $\mathcal{W}(a,b)$ \cite{FKL},
the thin Lie algebra \cite{FKL}, the
twisted Heisenberg--Virasoro algebra \cite{yh21},
the Schrödinger--Virasoro algebra \cite{yh21}, the
extended Schrödinger--Virasoro algebra \cite{yh21},
the $3$-dimensional Heisenberg Lie algebra \cite{yh21},
 Block Lie algebras and superalgebras \cite{kk22},
 Witt type algebras \cite{kk23},
 generalized Witt Lie algebras \cite{kk24},
 the Lie algebra of upper triangular matrices \cite{kk25},
 Lie incidence algebras \cite{kkinc},
 the Schrödinger algebra in $(n+1)$-dimensional space-time \cite{zzk}, and so on.
 The complete algebraic and geometric classification of complex $3$-dimensional transposed Poisson algebras was given in \cite{bk22}.

The paper is organized as follows.
In Section \ref{prem}, we recall some definitions and known results,
give some effective methods for describing transposed Poisson structures on given Lie algebras, and apply these methods to some special classes of Lie algebras. 
In Section \ref{rell},  we establish many relations between transposed Poisson algebras and other well-known structures,
such as generalized Poisson algebras,  Gelfand-Dorfman algebras,  $F$-manifold algebras, algebras of Jordan brackets, fields of fractions,  Poisson $n$-Lie algebras, etc.  In Section \ref{open},  we give a list of open questions for future research, with notes on the answers known to us.

\section{Transposed Poisson algebras and structures}\label{prem} All algebras and vector spaces are considered over the complex field, although many results can be proved similarly over other fields. The definition of a transposed Poisson algebra was given by Bai, Bai, Guo, and Wu \cite{bai20}. The concept of a $\frac{1}{2}$-derivation, as a particular case of the notion of a $\delta$-derivation, was introduced by Filippov \cite{fil1} (see also \cite{z10,k12} and the references therein).

\begin{definition}[see \cite{bai20}]\label{tpa}
Let ${\mathfrak L}$ be a vector space equipped with two nonzero bilinear operations $\cdot$ and $[\cdot,\cdot].$
The triple $({\mathfrak L},\cdot,[\cdot,\cdot])$ is called a transposed Poisson algebra if $({\mathfrak L},\cdot)$ is a commutative associative algebra,
$({\mathfrak L},[\cdot,\cdot])$ is a Lie algebra, and the following compatibility condition holds:
\begin{center}
$2z\cdot [x,y]=[z\cdot x,y]+[x,z\cdot y].$\end{center}
\end{definition}

	\begin{definition} 
		Let $({\mathfrak L},[\cdot,\cdot])$ be a Lie algebra. A  transposed Poisson (algebra) structure  on $({\mathfrak L},[\cdot,\cdot])$ is a commutative associative operation $\cdot$ on $\mathfrak L$ which makes $({\mathfrak L},\cdot,[\cdot,\cdot])$ a transposed Poisson algebra.
  	A transposed Poisson structure $\cdot$ on ${\mathfrak L}$ is called trivial if $x\cdot y=0$ for all $x,y\in{\mathfrak L}$.
	\end{definition}

\begin{definition}\label{12der}
Let $({\mathfrak L}, [\cdot,\cdot])$ be an algebra with a multiplication $[\cdot,\cdot]$. Let $\varphi$ be a linear map, and let
$\phi$ be a bilinear map.
Then $\varphi$ is a $\frac{1}{2}$-derivation if it satisfies
\begin{center}
$\varphi[x,y]= \frac{1}{2} \big([\varphi(x),y]+ [x, \varphi(y)] \big);$
\end{center}
 $\phi$ is a $\frac{1}{2}$-biderivation if it satisfies
\begin{longtable}{crl}
$\phi([x,y],z)$&$=$&$ \frac{1}{2} \big( 
 [\phi(x,z),y] +[x,\phi(y,z)]\big),$\\ 
$\phi(x,[y,z])$&$=$&$ \frac{1}{2} \big([\phi(x,y),z]+ 
[y,\phi(x,z)] \big).$
\end{longtable}
\end{definition}

Summarizing Definitions \ref{tpa} and \ref{12der}, we have the following key lemma.
\begin{lemma}
Let $({\mathfrak L},\cdot,[\cdot,\cdot])$ be a transposed Poisson algebra, 
and let $z$ be an arbitrary element from ${\mathfrak L}.$
Then the right multiplication $R_z$ in the commutative associative algebra $({\mathfrak L},\cdot)$ gives a $\frac{1}{2}$-derivation of the Lie algebra $({\mathfrak L}, [\cdot,\cdot])$, and
the bilinear map  $\cdot$  gives a
$\frac{1}{2}$-biderivation of $({\mathfrak L},[\cdot,\cdot]).$
\end{lemma}

The basic examples of $\frac{1}{2}$-derivations are the multiplications by elements of the ground field;
we call such $\frac{1}{2}$-derivations trivial.
The following theorem shows that only non-trivial $\frac{1}{2}$-derivations are of interest.

\begin{theorem}[see \cite{FKL}]\label{princth}
Let ${\mathfrak L}$ be a Lie algebra without non-trivial $\frac{1}{2}$-derivations.
Then every transposed Poisson structure defined on ${\mathfrak L}$ is trivial.
\end{theorem}

\subsection{Unital transposed Poisson algebras}\label{trgenpoi}
In the present subsection, we relate unital transposed Poisson algebras to generalized Poisson algebras, algebras of Jordan brackets, and quasi-Poisson algebras (see Section \ref{rell}).
Some theorems known for these types of algebras yield interesting results and conjectures in the theory of transposed Poisson algebras.

\begin{theorem}[see \cite{bai20}]\label{nonass}
Let $(\A, \cdot, [ \cdot ,\cdot  ] )$ be a complex transposed Poisson algebra. Let $(\A, \cdot)$ be a unital algebra.  
Then there is a derivation $\mathfrak{D}$ on $(\A, \cdot)$ such that
\begin{equation}\label{cbracket}
    [ x , y  ]= \mathfrak{D}(x)\cdot y-x\cdot \mathfrak{D}(y).
    \end{equation}
\end{theorem}


Each complex finite-dimensional semisimple commutative associative algebra is a direct sum of one-dimensional ideals; hence, it is unital, and all its derivations are zero.
Therefore, all transposed Poisson structures on a complex finite-dimensional semisimple commutative associative algebra $(\A, \cdot)$ are trivial.

\subsection{Transposed Poisson structures on nilpotent $n$-Lie algebras}
Let us recall the definitions of $n$-Lie algebras and transposed Poisson $n$-Lie algebras.

\begin{definition}
Let ${\mathfrak L}$ be a vector space equipped with an anticommutative $n$-linear operation $[\cdot ,\ldots, \cdot].$
Then $({\mathfrak L},[\cdot ,\ldots, \cdot])$ is called  an $n$-Lie algebra if
\begin{equation*}
 [[x_1 ,\ldots, x_n], y_2, \ldots, y_n] = \sum\limits_{i=1}^n
 [x_1,\ldots, [x_i, y_2,\ldots, y_n], \ldots, x_n]. 
\end{equation*}

\end{definition}

\begin{definition}[see \cite{bai20}]
Let ${\mathfrak L}$ be a vector space equipped with two nonzero operations: a bilinear operation $\cdot$ and an $n$-linear operation $[\cdot ,\ldots, \cdot].$
The triple $({\mathfrak L},\cdot,[\cdot ,\ldots, \cdot])$ is called a transposed Poisson $n$-Lie algebra if $({\mathfrak L},\cdot)$ is a commutative associative algebra,
$({\mathfrak L},[\cdot ,\ldots, \cdot])$ is an $n$-Lie algebra, and the following compatibility condition holds:
\begin{equation}
 z\cdot [x_1 ,\ldots, x_n]= \frac{1}{n} \sum\limits_{i=1}^n [x_1,\ldots, z\cdot x_i, \ldots, x_n]. 
 \label{tnp} 
\end{equation}

\end{definition}

\begin{theorem}\label{nilpotent}
Let $\mathfrak{L}$ be a non-abelian nilpotent $k$-dimensional $n$-Lie algebra ($n<k$).
Then $\mathfrak{L}$ admits a non-trivial transposed Poisson $n$-Lie structure.
\end{theorem}

\begin{proof}
Choose a basis $\{e_1,e_2,\ldots,e_k\}$ of $\mathfrak{L}$ in the following way:
$e_1,e_2, \ldots, e_n$ are linearly independent modulo $\mathfrak{L}^2$ (it is possible, since $\mathfrak{L}$ is non-abelian and nilpotent), $\mathfrak{L}^2\subseteq\langle e_{n+1},\ldots,e_k\rangle$ and $e_k \in {\rm Ann}(\mathfrak{L})\cap\mathfrak{L}^2.$
It is easy to see that the linear mapping $\varphi$ given by
    $\varphi(e_1)=\varphi(e_2)=\ldots=\varphi(e_n)=e_k$ and $\varphi(e_i)=0$ for $i>n$
is a $\frac{1}{n}$-derivation of $\mathfrak{L}.$
Then the commutative associative multiplication $\cdot$ given by 
    $e_i \cdot e_j =e_k$    $(1\leq i,j\leq n)$
gives a non-trivial structure of transposed Poisson $n$-Lie algebra.
\end{proof}

\subsection{Transposed Poisson structures and ${\rm Hom}$-Lie algebra structures}\label{hom}
Let us recall the definition of ${\rm Hom}$-Lie structures on Lie algebras.
\begin{definition}
Let $({\mathfrak L}, [\cdot,\cdot])$ be a Lie algebra, and let $\varphi$ be a linear map.
Then $({\mathfrak L}, [\cdot,\cdot], \varphi)$ is a ${\rm Hom}$-Lie structure on $({\mathfrak L}, [\cdot,\cdot])$ if 
\[
[\varphi(x),[y,z]]+[\varphi(y),[z,x]]+[\varphi(z),[x,y]]=0.
\]

\end{definition}

\medskip 

Let us consider some special cases of ${\rm Hom}$-Lie structures.

\begin{enumerate}
    \item[$\bullet$] 
$({\mathfrak L}, [\cdot,\cdot], \varphi)$ is a trivial ${\rm Hom}$-Lie structure on $({\mathfrak L}, [\cdot,\cdot])$ if $\varphi\in\mathbb C \cdot{\rm id}.$

\item[$\bullet$] $({\mathfrak L}, [\cdot,\cdot], \varphi)$ is a central ${\rm Hom}$-Lie structure on $({\mathfrak L}, [\cdot,\cdot])$ if $\varphi({\mathfrak L}) \subseteq Z({\mathfrak L}).$

\item[$\bullet$] $({\mathfrak L}, [\cdot,\cdot], \varphi)$ is a $\mathfrak{c}$-trivial ${\rm Hom}$-Lie structure on $({\mathfrak L}, [\cdot,\cdot])$ if $\varphi$ is a linear combination of trivial and central ${\rm Hom}$-Lie structures.
\end{enumerate}

By \cite[Proposition 2.9]{bai20}, for each transposed Poisson algebra $({\mathfrak L}, \cdot, [\cdot,\cdot])$ and each element $h\in{\mathfrak L}$, the map $\varphi_h(x)= h \cdot x$ defines a ${\rm Hom}$-Lie structure on $({\mathfrak L}, [\cdot,\cdot]).$ 
Hence, we have the following statement.

\begin{proposition}\label{homtr}
If a Lie algebra  $({\mathfrak L}, [\cdot,\cdot])$ does not admit non-trivial ${\rm Hom}$-Lie structures, then  $({\mathfrak L}, [\cdot,\cdot])$ does not admit non-trivial transposed Poisson structures.
\end{proposition}

The following algebras are known to have no non-trivial ${\rm Hom}$-Lie structures: 
\medskip 
\begin{enumerate}
\item complex finite-dimensional  simple Lie algebras ({\rm dim} $\neq 3$) and superalgebras \cite{xjl15,jl08,yl15};

\item infinite-dimensional Lie superalgebras of vector fields \cite{ysl14,sl17};

\item infinite-dimensional  Cartan algebras $W(n>1),$ $S(n),$ $H(2r),$ $K(2r + 1)$ \cite{xl17}.

\end{enumerate}
\medskip 

Hence, using Proposition \ref{homtr}, we have the following statement.

\begin{lemma}

Complex finite-dimensional  simple Lie algebras ({\rm dim} $\neq 3$) and superalgebras,
infinite-dimensional Lie superalgebras of vector fields, 
and  infinite-dimensional  Cartan algebras $W(n>1),$ $S(n),$ $H(2r),$ $K(2r + 1)$ do not  admit non-trivial transposed Poisson structures.

\end{lemma}

Let us give one more useful observation.

\begin{proposition}\label{homctr}
If a perfect Lie algebra  $({\mathfrak L}, [\cdot,\cdot])$ with one-dimensional center $Z({\mathfrak L})$ admits only  $\mathfrak c$-trivial ${\rm Hom}$-Lie structures, then  $({\mathfrak L}, [\cdot,\cdot])$ does not admit non-trivial transposed Poisson structures.
\end{proposition}

\begin{proof}
Each map $\varphi_h$ has the form $\alpha_h{\rm id}+\psi_h$, where $\psi_h({\mathfrak L})\subseteq Z({\mathfrak L})$. Then $\psi_h$ is a $\frac{1}{2}$-derivation with central values, hence $\psi_h([{\mathfrak L},{\mathfrak L}])=0$ and $\psi_h=0$, since ${\mathfrak L}$ is perfect. Now the statement follows from Theorem \ref{princth}.
\end{proof}

The condition that ${\mathfrak L}$ is perfect cannot be omitted: if $f$ is a non-zero linear form on ${\mathfrak L}$ with $f([{\mathfrak L},{\mathfrak L}])=0$ and $z\in Z({\mathfrak L})$, then $x\cdot y=f(x)f(y)z$ is a non-trivial transposed Poisson structure (for example, $c\cdot c=c$ on $\mathfrak{sl}_3\oplus\mathbb C c$).

By \cite{MZ18}, all ${\rm Hom}$-Lie structures on affine Kac–Moody algebras
are $\mathfrak c$-trivial. However, affine Kac–Moody algebras are not perfect, and Proposition \ref{homctr} cannot be applied. In fact, we have the following statement.

\begin{lemma}\label{kacmoody}
Complex affine Kac–Moody algebras admit non-trivial transposed Poisson structures.
\end{lemma}

\begin{proof}
Let $\widehat{\mathfrak g}=\mathfrak g\otimes\mathbb C[t,t^{-1}]\oplus\mathbb Cc\oplus\mathbb Cd$ be an affine Kac--Moody algebra, where $c$ is central and $d$ acts as $t\frac{d}{dt}$. Then $d\notin[\widehat{\mathfrak g},\widehat{\mathfrak g}]$, and the multiplication $d\cdot d=c$ (all other products of basis elements are zero) gives a non-trivial transposed Poisson structure.
\end{proof}

\subsection{Transposed Poisson structures 
on simple graded Lie algebras of finite growth}\label{growth}
Let us recall the definition of a complex simple graded Lie algebra of finite growth. A complex graded Lie algebra is
a Lie algebra $\mathfrak L$ endowed with a decomposition 
$\mathfrak L = \bigoplus\mathfrak L_i,$ such that $[\mathfrak L_i, \mathfrak L_j ] \subseteq \mathfrak L_{i+j}$
for every $i, j \in \mathbb{Z}.$ 
We always assume that every homogeneous component $\mathfrak L_i$ is
finite-dimensional. A graded Lie algebra $\mathfrak L$ is called simple graded if $\mathfrak L$ is not abelian and does not contain any non-trivial graded ideal. We say that $\mathfrak L$ is of finite growth
if the function $n \to {\rm dim} \  \mathfrak L_n$ is bounded by some polynomial. 
In \cite{Mat92}, Mathieu proved the following theorem,
which was conjectured by Kac.

\begin{theorem} 
Let $\mathfrak L$ be a graded simple Lie algebra of finite growth. Then $\mathfrak L$ is
isomorphic to one of the following Lie algebras:
\begin{enumerate}
    \item  a finite-dimensional simple Lie algebra;
    \item  a loop algebra; 
    \item  a Cartan algebra; 
    \item  the Witt algebra ${\rm W}$.
\end{enumerate}

\end{theorem}
   
Xie and Liu described all  
${\rm Hom}$-Lie structures on simple graded Lie algebras of finite growth \cite{xl17}.
Combining the results of \cite{Mat92} and \cite{xl17}, we obtain the following theorem.

\begin{theorem} \label{finitegrowth}
Let $\mathfrak L$ be a complex graded simple Lie algebra of finite growth.  If $\mathfrak L$ admits a non-trivial transposed Poisson structure, then it is isomorphic to one of the following Lie algebras:
\begin{enumerate}
    \item  the Witt algebra ${\rm W}$, or
    \item  the Cartan algebra ${\rm W}(1)$. 
\end{enumerate}
\end{theorem}

\begin{proof}
By \cite{xl17}, among these algebras, only
loop algebras, the Witt algebra ${\rm W}$, and the Cartan algebra ${\rm W}(1)$ have non-trivial ${\rm Hom}$-Lie structures.

\begin{enumerate}
\item[$\bullet$] It follows easily from the description of ${\rm Hom}$-Lie structures that loop algebras do not admit non-trivial transposed Poisson structures.

    \item[$\bullet$] All non-trivial transposed Poisson structures on the Witt algebra are described in \cite[Theorem 20]{FKL}.

    \item[$\bullet$] The Cartan algebra ${\rm W}(1)$ has the following
    multiplication table:
    \begin{center}
        $[e_i, e_j]=(i-j)e_{i+j}, \ \  i,j \geq -1.$
    \end{center}
    Using the same arguments as in the description of transposed Poisson structures on the Witt algebra \cite[Theorem 20]{FKL},
    one can show that each transposed Poisson structure on ${\rm W}(1)$
    is given by the multiplication
    $$e_i \cdot e_j =\sum\limits_{t >0} \alpha_t e_{i+j+t}$$
    for a finite set $\{ \alpha_t \}$. Conversely,
    each finite set $\{ \alpha_t \}$ gives a transposed Poisson structure in this way.
    
\end{enumerate}

\end{proof}

\begin{remark}\label{w1rem}
Let us note that the transposed Poisson structure on the Cartan algebra ${\rm W}(1)$ defined by $\{ \alpha_t \}$ is unital if and only if $\alpha_1\neq0$ and $\alpha_t=0$ for $t>1$; for example, $e_i\cdot e_j=e_{i+j+1}$ is a unital transposed Poisson structure on ${\rm W}(1)$ with the unit $e_{-1}$.

\end{remark}
 
\section{Relations with other  algebras}\label{rell}
\subsection{Quasi-Poisson algebras}\label{qp}
Billig introduced quasi-Poisson algebras \cite{billig}. 

\begin{definition}[see \cite{billig}]
Let $(\A, \cdot)$ be a unital commutative associative algebra, let
$(\A, \{ \cdot ,\cdot\})$ be a Lie algebra, and let $\mathfrak D$ be a derivation of $(\A, \cdot)$.
 Then $(\A, \cdot, \{ \cdot ,\cdot\})$ is called a quasi-Poisson algebra if

\begin{center}
    $a \cdot   \big(\mathfrak D(\{b, c\}) + \{b, c\}\big) = 
    \{a \cdot (\mathfrak D(b) + b), c\} + 
    \{b, a \cdot (\mathfrak D(c) + c)\} + 
    \{a, b\} \cdot (\mathfrak D(c) + c) - (\mathfrak D(b) + b) \cdot \{a, c\}.$
\end{center}

\end{definition} 
The main example of quasi-Poisson algebras is given by relation (\ref{cbracket})   \cite[Lemma  6.3]{billig}.

\begin{corollary}
Let $(\A,\cdot, [\cdot,\cdot])$ be a unital transposed Poisson algebra. Then
$(\A,\cdot, [\cdot,\cdot])$ is a quasi-Poisson algebra
with the derivation $\mathfrak D$ from Theorem \ref{nonass}.
\end{corollary}

\subsection{Generalized Poisson algebras and algebras of   Jordan brackets}
Let us recall the definitions of generalized Poisson brackets and Jordan brackets \cite{ck07,k17}.
\begin{definition}
Let $(\A, \cdot)$ be a unital commutative associative algebra. 
\begin{enumerate}
    \item[$\bullet$] An anticommutative product $\{\cdot,\cdot\}: \A \otimes \A \to \A$ is called a generalized Poisson bracket (it is also known as the contact bracket   \cite{z21}) if 
$(\A, \{\cdot, \cdot\})$ is a Lie algebra and 
\begin{longtable}{rcl}
$\{x,y\cdot z\}$&$=$&$\{x,y\}\cdot z + y\cdot \{x,z\}-\{x,1\} \cdot y \cdot z.$
\end{longtable}
  \item[$\bullet$] An anticommutative product $\{\cdot,\cdot\}: \A \otimes \A \to \A$ is called a Jordan bracket (see \cite{k17}) if 
\begin{longtable}{rcl} 
$\{x,y\cdot z\}$&$=$&$\{x,y\}\cdot z + y\cdot \{x,z\}-\{x,1\} \cdot y \cdot z,$\\
$\{x,\{y,z\}\} + \{y,\{z,x\}\} + \{z,\{x, y\}\}$&$ =$&$ \{x,1\}\cdot\{y,z\} + \{y,1\}\cdot\{z, x\}  +\{z,1\}\cdot\{x, y\}.$
\end{longtable}
\end{enumerate}
\end{definition}
The main example of generalized Poisson brackets and  Jordan brackets is given by relation (\ref{cbracket}). Hence, we have the following corollary.

\begin{corollary}\label{transcontact}
Let $(\A,\cdot, [\cdot,\cdot])$ be a unital transposed Poisson algebra. Then
$(\A,\cdot, [\cdot,\cdot])$ is a generalized Poisson algebra and an algebra of Jordan brackets.
\end{corollary}

Martínez, Shestakov, and Zelmanov have shown that all unital Jordan brackets are embeddable into Poisson
brackets \cite{msz01}.
Combining Corollary \ref{transcontact} with this result, we obtain the following statement. 

\begin{corollary}\label{transcontactemb}
Unital transposed Poisson brackets are embeddable into Poisson
brackets.
\end{corollary}

\subsection{Algebras of Jordan brackets}\label{jordbr}
Let $(\A, \cdot)$ be a commutative associative algebra, and 
let $(\A, \{\cdot, \cdot\})$ be an anticommutative algebra.
Then $(\A, \cdot, \{\cdot, \cdot\})$ is called an algebra of Jordan brackets if the Kantor double of $(\A, \cdot, \{\cdot, \cdot\})$ is a Jordan superalgebra \cite{kantor92}.
If $(\A, \cdot)$ is not unital, then this condition is equivalent to the following identities (see \cite{k10}): 

{\small
\begin{longtable}{rcl}
    
 $\{\{x , y\}\cdot  z , t\} + \{\{y, t\} \cdot  z , x\}+ \{\{t, x\} \cdot z , y\} $&$=$&$ 
 \{x , y\} \cdot \{z , t\}+ \{y, t\} \cdot \{z , x\}  + \{t , x\} \cdot \{z , y\}, $\\
 
$ \{y \cdot t , z\}\cdot x +\{x,z\}\cdot  y\cdot t $&$=$&$ \{t\cdot x , z\}\cdot y + \{y,z\}\cdot t\cdot x,$\\ 
 
 $\{t\cdot x , y\cdot z\}+\{t\cdot y , x\cdot z\}+\{x \cdot y\cdot z , t\}$&$ =
 $&$\{ t\cdot y , z\}\cdot x   + \{t\cdot x , z\}\cdot y  + x\cdot y\cdot \{z , t\}.$

\end{longtable}}
It is easy to see (for instance, using the identities from \cite[Theorem 2.5]{bai20}) 
that each transposed Poisson algebra is 
an algebra of Jordan brackets (see also \cite{fer23}).
Recall that Poisson algebras have the same property.
Hence, the variety of algebras of Jordan brackets is the first \textquotedblleft interesting\textquotedblright \ (i.e., well-known) variety of algebras that contains all Poisson and all transposed Poisson algebras.

\begin{corollary}
The Kantor double of a transposed Poisson algebra is a Jordan superalgebra.
\end{corollary}

\subsection{Gelfand--Dorfman algebras (GD-algebras) and 
 $F$-manifold algebras}\label{fma}

 \begin{definition}
 A vector space $\A$ with a bilinear product $\cdot$ is called a Novikov algebra if 
 \begin{longtable}{rcl}
$(x \cdot y) \cdot z-  x\cdot  (y \cdot z)$&$ =$&$  
(y\cdot x) \cdot z - y \cdot  (x \cdot z),$\\
$(x \cdot y) \cdot z$ &$=$&$ (x\cdot z) \cdot y.$
\end{longtable}
\end{definition}

 \begin{definition}
 Let $(\A, \cdot)$ be a  Novikov algebra, 
 and let
$(\A, [ \cdot ,\cdot] )$ be a Lie algebra.
 Then $(\A, \cdot, [ \cdot ,\cdot] )$
 is called a Gelfand--Dorfman algebra if 
\begin{longtable}{rcl}
$[x, y \cdot z] - [z, y \cdot x] + [y, x] \cdot z - [y, z] \cdot x - y \cdot [x, z]$&$ =$&$ 0.$
\end{longtable} 
\end{definition}

 \begin{definition}
 Let $(\A, \cdot)$ be a  commutative associative  algebra, 
 and let
$(\A, [ \cdot ,\cdot])$ be a Lie algebra.
 Then $(\A, \cdot, [ \cdot ,\cdot])$
 is called an $F$-manifold algebra if 
\begin{longtable}{lcl}
$[x \cdot y, z \cdot t]$&$ =$&$ [x \cdot y, z] \cdot t + 
[x \cdot y, t] \cdot z + x \cdot [y, z \cdot t] + y \cdot [x, z \cdot t]-$\\ 
&&$\big(x \cdot z \cdot [y, t] + y \cdot z \cdot [x, t] + 
y \cdot t \cdot [x, z] + x \cdot t \cdot [y, z]\big).$
\end{longtable} 
\end{definition}

Sartayev proved that the variety of transposed Poisson algebras coincides with the variety of commutative GD-algebras (i.e., GD-algebras with commutative $\cdot$) and
that each commutative GD-algebra is an $F$-manifold algebra \cite{kms}.
Note that each Poisson algebra is also an $F$-manifold algebra.
Hence, the variety of $F$-manifold algebras is the second \textquotedblleft interesting\textquotedblright \ variety of algebras that contains all Poisson and all transposed Poisson algebras.

\subsection{Transposed Poisson PI algebras}
The celebrated Amitsur-Levitsky theorem states that
the algebra
${\rm Mat}_k(R)$
of
$k \times k$
matrices over a commutative ring~%
$R$
satisfies the identity
$s_{2k}=0,$
where $s_m$ is the standard polynomial
\begin{center}
    $s_m(x_1, \ldots, x_m)=\sum_{\sigma \in \mathbb{S}_m} (-1)^{\sigma} x_{\sigma(1)} \ldots x_{\sigma(m)}.$ 
\end{center}
Furthermore,
every associative PI algebra satisfies
$(s_k)^l=0$
by Amitsur's theorem.
Farkas defined customary polynomials
\begin{center} 
$g=
\sum_{\sigma \in {\mathbb S}_{m}} c_{\sigma} \{x_{\sigma(1)}, x_{\sigma(2)} \}\cdot\ldots\cdot \{ x_{\sigma(2i-1)} , x_{\sigma(2i)} \}$
\end{center}
and proved that
every Poisson PI algebra satisfies some customary identity \cite{F98}. 
Farkas' theorem was extended to generic Poisson algebras in \cite{klms14}. 

Kaygorodov found an~analog 
of customary identities for unital generalized Poisson algebras and unital algebras of Jordan brackets in \cite{k17}:
\begin{eqnarray}
\label{tojdest}
{\mathfrak g}=
\sum\limits_{i=0}^{[m/2]}\sum_{\sigma \in {\mathbb S}_{m}} c_{\sigma,i} \langle x_{\sigma(1)}, x_{\sigma(2)}\rangle \cdot  \ldots \cdot
\langle x_{\sigma(2i-1)} , x_{\sigma(2i)} \rangle \cdot {\mathfrak D}(x_{\sigma(2i+1)})\cdot \ldots \cdot  {\mathfrak D}(x_{\sigma(m)}),
\end{eqnarray}
where 
$
\langle x,y\rangle :=\{x,y\} -\big({\mathfrak D}(x)\cdot y-x\cdot {\mathfrak D}(y)\big)
\mbox{ and } {\mathfrak D} \mbox{ is the related derivation, i.e.,  } {\mathfrak D}(x)=\{x,1\}.$

\begin{theorem}[Theorem 13, \cite{k17}] 
If a~unital generalized Poisson algebra~%
$\A$
satisfies a~polynomial identity,
then~%
$\A$
satisfies a~polynomial identity
${\mathfrak g}$
of type (\ref{tojdest}).
\end{theorem}

Consequently,
we obtain the following statement.

\begin{theorem}\label{farktr}
\footnote{Dotsenko, Ismailov, and Umirbaev proved an analog of the present theorem \cite{diu}: if a unital transposed Poisson algebra satisfies a $[\cdot,\cdot]$-free polynomial identity (i.e., an identity involving only the multiplication $\cdot$), then it satisfies an identity of the type
${\mathfrak D}(x_1)\cdot\ldots \cdot {\mathfrak D}(x_m)=0.$
These results motivate the following interesting question.

{\bf Open question}.
Let $\bf A$ be a commutative associative algebra with a derivation $\mathfrak D$ satisfying a differential polynomial identity
(i.e., a functional identity involving the derivation $\mathfrak D$).
Is there a special type of identities for $\bf A$?
} 
If a~unital transposed Poisson algebra~%
$\A$
satisfies a~polynomial identity,
then~%
$\A$
satisfies polynomial identities ${\mathfrak g}$ and ${\mathfrak f}$ of the following types
\begin{center}
    ${\mathfrak g}(x_1, \ldots, x_m)=
\sum_{\sigma \in {\mathbb S}_{m}} c_{\sigma}  {\mathfrak D}(x_{\sigma(1)})\cdot\ldots \cdot {\mathfrak D}(x_{\sigma(m)}),$
\end{center}
 where ${\mathfrak D}$ is a derivation,
 or 
\begin{center}
${\mathfrak f}(x_1, \ldots, x_{3m}) =
\sum_{\sigma \in \mathbb{S}_{m}} c_{\sigma}  
 \prod\limits_{k=1}^{m} 
\Big(  [ x_{\sigma(k)} \cdot x_{m+2k-1}, x_{m+2k} ] +[ x_{\sigma(k)} \cdot x_{m+2k}, x_{m+2k-1} ]  \Big).$\end{center} 
\end{theorem}
\begin{proof}
In our case, $\langle x,y\rangle=0,$
and identity (\ref{tojdest}) gives the identity ${\mathfrak g}$ from the statement.
Now multiply the identity ${\mathfrak g}$ by $\prod\limits_{k=m+1}^{3m} x_k$ and apply the following relation, which is a consequence of the  transposed Poisson identity and the generalized Poisson identity,
\begin{center}
 ${\mathfrak D}(x) \cdot y \cdot z= \frac{1}{2} \big([ x \cdot y, z ] +[ x  \cdot z, y ]\big).$
\end{center}
Hence, the identity  ${\mathfrak g}$ gives the identity  ${\mathfrak f}.$
\end{proof}

\subsection{Transposed Poisson fields}\label{tpfields}
Denote by $\mathcal{Q}(\mathfrak{P})$ the field of fractions of a commutative associative algebra $\mathfrak{P}.$ A Poisson bracket $\{\cdot,\cdot\}$ on a domain $\mathfrak{P}$ can be extended (see, for example, \cite{kaledin,mlu16,makar12}) to a Poisson bracket on $\mathcal{Q}(\mathfrak{P})$ by
\[\left\{ \frac{a}{b}, \frac{c}{d} \right\} \  = \
\frac{\{a,c\}\cdot b\cdot d -\{a,d\}\cdot b\cdot c - \{b,c\}\cdot a\cdot d+\{b,d\}\cdot a\cdot c}{b^2\cdot d^2}.\]
Our aim in the present subsection is to give a similar construction in the transposed Poisson case.    
 
\begin{theorem}\label{tpfield}
Let $(\mathfrak{L}, \cdot, [\cdot,\cdot])$ be a unital transposed Poisson algebra such that $(\mathfrak{L},\cdot)$ has no zero divisors.
Then $( \mathcal{Q}(\mathfrak{L}), \bullet, \llbracket \cdot,\cdot \rrbracket),$
where 
\[  \frac{a}{b} \bullet \frac{c}{d} \ = \ \frac{a\cdot c}{b\cdot d}    \mbox{\ \ and \ \ } 
\left\llbracket \frac{a}{b}, \frac{c}{d} \right\rrbracket \ = \ \frac{[a,b]\cdot c\cdot d-a\cdot b\cdot [c,d]}{b^2\cdot d^2}\]
is a  transposed Poisson field.

\end{theorem}
\begin{proof}
In fact, it is easy to see that $ \bullet$ is a commutative associative multiplication.
Thanks to Theorem \ref{nonass}, our bracket $[x,y]$ 
can be written as $[x,y]=\mathfrak{D}(x)\cdot y-x\cdot \mathfrak{D}(y)$ for a derivation $\mathfrak{D}$ of $(\mathfrak{L},\cdot)$.
The derivation $\mathfrak{D}$ extends uniquely to a derivation of the field $\mathcal{Q}(\mathfrak{L})$ by
$\mathfrak{D}\big(\frac{a}{b}\big)=\frac{\mathfrak{D}(a)\cdot b-a\cdot\mathfrak{D}(b)}{b^2}=\frac{[a,b]}{b^2}$, and
\[\left\llbracket \frac{a}{b}, \frac{c}{d} \right\rrbracket=\mathfrak{D}\Big(\frac{a}{b}\Big)\bullet\frac{c}{d}-\frac{a}{b}\bullet\mathfrak{D}\Big(\frac{c}{d}\Big).\]
Since every commutative associative algebra with a derivation $\mathfrak{D}$ gives a transposed Poisson algebra with the bracket $\mathfrak{D}(x)y-x\mathfrak{D}(y)$ \cite{bai20}, the statement follows.
\end{proof}

\subsection{Transposed Poisson $n$-Lie algebras and Poisson $n$-Lie algebras}
The following result is a direct generalization of a result of \cite{bai20} to the $n$-ary case.

\begin{definition}
Let ${\mathfrak L}$ be a vector space equipped with two nonzero operations -- a bilinear operation $\cdot$ and an $n$-linear operation $[\cdot ,\ldots, \cdot].$
The triple $({\mathfrak L},\cdot,[\cdot ,\ldots, \cdot])$ is called a Poisson $n$-Lie algebra if $({\mathfrak L},\cdot)$ is a commutative associative algebra,
$({\mathfrak L},[\cdot ,\ldots, \cdot])$ is an $n$-Lie algebra, and the following compatibility condition holds:
\begin{equation}
 [x \cdot y , z_2, \ldots, z_n]=   x \cdot [y, z_2, \ldots, z_n]+[x, z_2, \ldots, z_n]  \cdot y. 
 \label{np} 
\end{equation}

\end{definition}

\begin{proposition}
Let $(\mathfrak L, \cdot)$ be a commutative associative algebra, and let $(\mathfrak L, [ \cdot , \ldots, \cdot])$ be an $n$-Lie algebra. Then $(\mathfrak L, \cdot,  [ \cdot,\ldots , \cdot ])$ is both a Poisson $n$-Lie and a transposed Poisson $n$-Lie algebra if and only if 
\begin{center}
    $x\cdot [y_1, \ldots, y_n] = [x\cdot y_1, \ldots, y_n] = 0.$    
\end{center}
\end{proposition} 

\begin{proof}
Let $x, y_1, \ldots, y_n \in \mathfrak L$. 
Substituting Eq. (\ref{tnp}) into Eq. (\ref{np}), we have 
\begin{longtable}{lcl}
    $n x\cdot [y_1, \ldots, y_n]$ &$=$& $ \sum\limits_{i=1}^n  [y_1,\ldots, x \cdot y_i, \ldots, y_n]$\\
    & $=$& $ \sum\limits_{i=1}^n  x\cdot [y_1, \ldots, y_n]+\sum\limits_{i=1}^n  y_i\cdot [y_1, \ldots, y_{i-1},x,y_{i+1}, \ldots, y_n].$ 
\end{longtable}
Substituting Eq. (\ref{np}) again into the above equality, we have 
\begin{longtable}{lclcl}
$0$ &$=$ &    $\sum\limits_{i=1}^n  y_i\cdot [y_1, \ldots, y_{i-1},x,y_{i+1}, \ldots, y_n]$&$ $ &\\

$ $ &$=$ &     $\sum\limits_{i=1}^n 
    \big(\sum\limits_{t=1}^{i-1}  [y_1, \ldots y_i\cdot y_t, \ldots, y_{i-1},x,y_{i+1}, \ldots, y_n]\big)$  & $+$ &\\
  
$ $ &$ $ &    \multicolumn{1}{r}{$\sum\limits_{i=1}^n 
  [y_1, \ldots  y_{i-1},y_i\cdot x,y_{i+1},  \ldots, y_n] $} & $+$ &\\

 $ $ &$ $ &   $\sum\limits_{i=1}^n 
    \big(\sum\limits_{t=i+1}^{n}  [y_1, \ldots  y_{i-1},x,y_{i+1}, \ldots, y_i\cdot y_t, \ldots, y_n]\big)$ & $ $ &\\
    
&  $=$  &    \multicolumn{1}{r}{$\sum\limits_{i=1}^n 
  [y_1, \ldots  y_{i-1},y_i\cdot x,y_{i+1},  \ldots, y_n] $} & $=$ & $n x \cdot  [y_1, \ldots  , y_n].$ \\

\end{longtable}
Hence $x \cdot  [y_1, \ldots  , y_n]=0$ and $[x \cdot  y_1, \ldots  , y_n]=0.$  

\end{proof}

\section{Open questions}\label{open} 

Below we compile a list of open questions for future research related to transposed Poisson structures. Footnotes to the questions describe the partial and complete answers known to us.

\begin{question}[Pasha Zusmanovich]
\label{q1}\footnote{Transposed Novikov--Poisson algebras were introduced and studied by Jin and Hong \cite{jh26}.}
Formulate the notion of  \textquotedblleft transposed\textquotedblright \ for any variety of algebras/operads with two $n$-ary operations. Can this be done in operadic or categorical language? This general notion should explain the appearance of \textquotedblleft scaling by 2\textquotedblright \ in the definition of transposed Poisson algebras. What are the transposed analogs of dendriform, post-Lie, Novikov--Poisson, etc., algebras?
What are self-transposed algebras/operads?
\end{question}

By \cite[Theorem 8]{FKL}, if a Lie algebra does not admit non-trivial $\frac{1}{2}$-derivations, then it admits neither non-trivial $\frac{1}{2}$-biderivations nor non-trivial transposed Poisson structures.
All algebras with non-trivial $\frac{1}{2}$-derivations considered in the present paper and in \cite{FKL,yh21} also admit non-trivial $\frac{1}{2}$-biderivations and non-trivial transposed Poisson structures.
These results motivate the following question.

\begin{question}\footnote{Kaygorodov and Khudoyberdiyev constructed a $6$-dimensional perfect non-semisimple Lie algebra with non-trivial $\frac{1}{2}$-derivations and without non-trivial transposed Poisson structures \cite{kksp}.}
Is there a Lie algebra that admits non-trivial $\frac{1}{2}$-derivations and does not admit non-trivial $\frac{1}{2}$-biderivations?
Is there a Lie algebra that admits non-trivial $\frac{1}{2}$-derivations and $\frac{1}{2}$-biderivations, but does not admit non-trivial transposed Poisson structures?
\end{question}

One of the first examples of non-trivial transposed Poisson structures was given on the Witt algebra \cite[Theorem 20]{FKL}.
The central extension of the Witt algebra is
the Virasoro algebra, which has no non-trivial
$\frac{1}{2}$-derivations,
$\frac{1}{2}$-biderivations, or transposed Poisson structures \cite[Theorem 27]{FKL}.
By \cite{xl17}, the infinite-dimensional Cartan algebras $S_n$ and $H_n$ have no
 non-trivial transposed Poisson structures
(or $\frac{1}{2}$-derivations, $\frac{1}{2}$-biderivations).
It is easy to see that if a non-perfect algebra has a non-zero annihilator, then it admits
non-trivial transposed  Poisson structures
(and $\frac{1}{2}$-derivations, $\frac{1}{2}$-biderivations): if $f$ is a non-zero linear form vanishing on $[{\mathfrak L},{\mathfrak L}]$ and $z\in{\rm Ann}({\mathfrak L})$, then $x\cdot y=f(x)f(y)z$ is such a structure.
For perfect algebras, the condition that the dimension of the annihilator is greater than $1$ is not sufficient: the $10$-dimensional perfect Lie algebra $\mathfrak{sl}_2\ltimes(V\otimes\mathbb C^2\oplus S^2(\mathbb C^2))$, where $V$ is the $2$-dimensional irreducible $\mathfrak{sl}_2$-module, $S^2(\mathbb C^2)$ is central and $[v\otimes a,w\otimes b]=\omega(v,w)\,ab$ for the invariant skew-symmetric form $\omega$ on $V$, has a $3$-dimensional center and only trivial $\frac{1}{2}$-derivations.
Dzhumadildaev proved  that the dimensions of annihilators of central extensions of  $S_n$ and $H_n$
are greater than $1$ in \cite{dzhuma92}.
These results motivate the following question.

\begin{question}
Is there a Lie algebra that does not admit non-trivial transposed Poisson structures
($\frac{1}{2}$-derivations, $\frac{1}{2}$-biderivations), 
but whose (indecomposable) one-dimensional central extension admits them?
\end{question}

 There are no non-trivial transposed Poisson structures on complex finite-dimensional semisimple Lie algebras \cite[Corollary 9]{FKL}. Moreover,
 the
Schrödinger algebra, which is the semidirect product of a simple and a nilpotent algebra, also has no non-trivial transposed Poisson structures (other perfect non-semisimple Lie algebras without non-trivial $\frac{1}{2}$-derivations, namely, Galilean algebras, can be found in \cite{klz22}).
By Theorem \ref{nilpotent}, each finite-dimensional nilpotent Lie algebra has a non-trivial transposed Poisson structure (and non-trivial $\frac{1}{2}$-derivations and $\frac{1}{2}$-biderivations).
By \cite{klz22}, each finite-dimensional solvable non-nilpotent Lie algebra has a non-trivial transposed Poisson structure (and non-trivial $\frac{1}{2}$-derivations and $\frac{1}{2}$-biderivations).
 These results motivate the following question.

\begin{question}\footnote{Kaygorodov and Khudoyberdiyev constructed a $6$-dimensional perfect non-semisimple Lie algebra with non-trivial $\frac{1}{2}$-derivations \cite{kksp}.}
\footnote{Yang, Tang, and Khudoyberdiyev constructed a $9$-dimensional non-perfect non-solvable Lie algebra with non-trivial $\frac{1}{2}$-derivations and non-trivial transposed Poisson structures \cite{zzk}.}
Is there a non-perfect  (non-solvable) Lie algebra that does not admit non-trivial transposed Poisson structures 
($\frac{1}{2}$-derivations, $\frac{1}{2}$-biderivations)? 
Is there a perfect  (non-semisimple) Lie algebra that   admits non-trivial transposed Poisson structures 
($\frac{1}{2}$-derivations, $\frac{1}{2}$-biderivations)? 
\end{question}

The description of $\frac{1}{2}$-derivations is a good tool for detecting Lie algebras without non-trivial transposed Poisson structures.
Another sufficient condition is given by ${\rm Hom}$-Lie structures (Proposition \ref{homtr}).
However, this condition is not necessary:
the simple $3$-dimensional algebra $\mathfrak{sl}_2$ has non-trivial ${\rm Hom}$-Lie structures \cite{xjl15}, but it does not admit non-trivial transposed Poisson structures \cite[Corollary 9]{FKL}.
 Filippov proved that each $\delta$-derivation ($\delta\neq0,1$) gives a non-trivial ${\rm Hom}$-Lie algebra structure \cite[Theorem 1]{fil1}. These results motivate the following question.

\begin{question}
Describe the Lie algebras on which each ${\rm Hom}$-Lie structure gives a $\frac{1}{2}$-derivation.
Which types of ${\rm Hom}$-Lie structures can guarantee the existence of non-trivial transposed Poisson structures on a Lie algebra?
\end{question}

Denote by $\Delta_{\frac{1}{2}}(\mathfrak L)$ the associative algebra generated by the $\frac{1}{2}$-derivations of $\mathfrak L$. The Witt algebra ${\rm W}$ has the following interesting property:
$ \mathfrak{Der} (\Delta_{\frac{1}{2}}({\rm W})) \cong {\rm W}$ \cite[Theorem 19]{FKL}. These results motivate the following question.

\begin{question}
Is there a Lie algebra $\mathfrak L$ ($ \not\cong \rm W $) such that $ \mathfrak{Der} (\Delta_{\frac{1}{2}}(\mathfrak L)) \cong \mathfrak L$?

\end{question}

By
 \cite[Corollary 9]{FKL}, there are no non-trivial transposed Poisson structures on
complex finite-dimensional semisimple Lie algebras;
by Theorem \ref{nonass}, there are no non-trivial transposed Poisson structures on
complex finite-dimensional semisimple commutative associative algebras.
On the other hand, there are simple Lie algebras (the Witt algebra ${\rm W}$ and the Cartan algebra ${\rm W}(1)$) 
that admit non-trivial transposed Poisson structures,
but, in these cases, the associative part is not simple (Theorem \ref{finitegrowth}). These results motivate the following question.

\begin{question}
Is there a (not necessarily commutative) transposed Poisson algebra $(\mathfrak L, \cdot, [\cdot, \cdot])$ such that 
its Lie and associative parts are simple (semisimple, prime, semiprime) algebras?
\end{question} Combining Theorem \ref{nonass} and \cite[Proposition 28]{FK21}, we see that the Kantor product of multiplications of a unital transposed Poisson algebra gives a new transposed Poisson algebra.
Furthermore, the Kantor product of transposed Poisson algebras constructed on the Witt algebra gives new transposed Poisson algebras \cite[Proposition 27]{FK21}.
These results motivate the following question.

\begin{question}
Is there a transposed Poisson algebra such that the algebra constructed by the Kantor product of multiplications (see \cite{FK21}) is not a transposed Poisson algebra?
\end{question}

The Kantor double gives a way to construct Jordan superalgebras from Poisson algebras and from
algebras of Jordan brackets \cite{kantor92}.
By Corollary \ref{transcontact},
the Kantor double of each unital transposed Poisson algebra is a Jordan superalgebra.
 Moreover, the Kantor double of the commutator algebra of a Novikov--Poisson algebra is a Jordan superalgebra \cite{zz}.
By \cite{fer23}, the Kantor double of each transposed Poisson superalgebra is a Jordan superalgebra.
These results motivate the following question.

\begin{question}
Describe and study the class of superalgebras (= a subvariety of Jordan superalgebras) 
obtained from transposed Poisson algebras by the Kantor double process.
\end{question}

The celebrated Amitsur–Levitsky theorem states that the algebra ${\rm Mat}_k$ of $k \times k$ matrices
over a commutative ring ${\rm R}$ satisfies the identity $s_{2k} = 0.$ 
Furthermore, every
associative PI algebra satisfies $(s_k)^l = 0$ by Amitsur’s theorem. 
Farkas defined customary Poisson polynomials and proved that every Poisson PI algebra satisfies some customary identity \cite{F98}. 
Farkas' theorem was extended to generic Poisson algebras in \cite{klms14},
to unital generalized Poisson algebras in \cite{k17},
and to unital transposed Poisson algebras (Theorem \ref{farktr}).
These results motivate the following question.

\begin{question}
Is there a special type of polynomial identities for (not necessarily unital)  transposed Poisson algebras satisfying a polynomial identity?
\end{question}

In 1972, Regev proved that the tensor product of two associative PI algebras is an associative PI algebra.
A similar result was obtained for Poisson algebras \cite{mpr07}.
These results motivate the following question.

\begin{question}
Is the tensor product of two transposed Poisson PI algebras a transposed Poisson PI algebra?
\end{question}

It is known that the tensor product of two Poisson algebras gives a Poisson algebra \cite{mpr07}, 
the tensor product of two transposed Poisson algebras gives a transposed Poisson algebra \cite{bai20},
the tensor product of two  pre-Lie Poisson algebras gives a   pre-Lie Poisson algebras \cite{bai20},
the tensor product of two Novikov-Poisson algebras gives a Novikov-Poisson algebra \cite{xu97}, 
and 
the tensor product of two anti-pre-Lie Poisson algebras gives an anti-pre-Lie Poisson algebra \cite{bai22}.
For $F$-manifold algebras, this is not true for the standard multiplication on the tensor product \cite{lsb21}.
Similarly, the tensor product of two generalized Poisson algebras is not necessarily a generalized Poisson algebra \cite{z21}.
These results motivate the following question.

\begin{question}[Pasha Zusmanovich]
Find identities for algebras formed by the tensor product of 
a transposed Poisson algebra and a Poisson algebra.
\end{question}

As mentioned in Sections \ref{jordbr} and \ref{fma},
all Poisson and all transposed Poisson algebras are algebras of Jordan brackets and
$F$-manifold algebras.
Hence, Poisson and transposed Poisson algebras lie in the intersection of
the varieties of $F$-manifold algebras and algebras of Jordan brackets.
These results motivate the following question.

\begin{question}
Characterize the variety of algebras obtained as the intersection of $F$-manifold algebras and algebras of Jordan brackets.
Find a minimal variety of algebras that contains 
Poisson and transposed Poisson algebras.
\end{question}

The notion of a depolarized  Poisson algebra 
(i.e., a binary algebra such that   the commutator and the anticommutator products give a Poisson algebra) was introduced and studied in \cite{MR06,gr08}. These results motivate the following question.

\begin{question}\footnote{Askar Dzhumadildaev proved that weak Leibniz algebras 
give the depolarization of transposed Poisson algebras \cite{dzhuma}.}
Define and study the variety of depolarized transposed Poisson algebras.
\end{question}

In 1950, Hall constructed a basis for 
free Lie algebras. 
Later,
Shestakov constructed a basis for free Poisson algebras;
Shestakov and Zhukavets constructed a basis for the
free Poisson--Malcev superalgebra with one generator;
Kaygorodov, Shestakov, and Umirbaev constructed a basis for free generic Poisson algebras in  \cite{ksu18};
and Kaygorodov constructed a basis for free unital generalized Poisson algebras in \cite{k17}. These results motivate the following question.

\begin{question}\footnote{Roughly speaking, a basis of a free Poisson algebra is a basis of a polynomial ring on a free Lie algebra. Our first impression (Chengming Bai mentioned, in one of his talks, that it was also the first impression of his team) was that a basis of a free transposed Poisson algebra can be found as a free Lie algebra on a polynomial ring, but this is not true.}
\unskip\footnote{Kaygorodov and Zhang constructed a basis of the free one-generated transposed Poisson algebra (it is indexed by partitions) and bases of the free transposed Poisson superalgebras with at most three odd generators. The study of free transposed Poisson superalgebras with odd generators led them to the following conjecture: for every $n$, the ideal of the free transposed Poisson superalgebra with $n$ odd generators generated by all products of its elements is finite-dimensional. For $n=1,\ldots,8$, the dimensions of this ideal are $0,$ $4,$ $56,$ $640,$ $7200,$ $81984,$ $945280,$ and $11008768$ (for $n\geq4$, they were computed with computer assistance). A basis of free metabelian transposed Poisson algebras was constructed by Abdukhalikov and Sartayev \cite{as25}.}
Construct a basis for free transposed Poisson algebras. \label{q15}
\end{question}

In 1930, Magnus proved one of the most important theorems of combinatorial group theory.
Let $G=\langle x_1,x_2, \ldots,x_n \ |\ r=1\rangle$ 
be a group defined by a single cyclically reduced relator $r.$ 
If $x_n$ appears in $r,$ then the subgroup of $G$ generated by 
$x_1, \ldots ,x_{n-1}$ is a free group, freely generated by $x_1,\ldots, x_{n-1}$.
He called it the Freiheitssatz and used it to give several applications, the decidability of the word problem for groups with a single defining relation among them.
 The Freiheitssatz was studied in various varieties of groups and algebras: Shirshov established it for Lie algebras; Romanovskii studied it for solvable and nilpotent groups; Makar-Limanov proved it for associative algebras over a field of characteristic zero.
 It was also proved for Novikov and right-symmetric algebras,
 as well as
 for Poisson algebras \cite{mlu09} and generic Poisson algebras \cite{klms14}. These results motivate the following question.

\begin{question}\label{q16}\footnote{Kaygorodov and Zhang gave a negative answer: the Freiheitssatz fails already for the free two-generated transposed Poisson algebra.}
Is the  Freiheitssatz true for transposed Poisson algebras?
\end{question}

In 1942,  Jung proved that all automorphisms of the polynomial algebra in two variables are tame. Later, in 1970, Makar-Limanov established the same result for the free associative algebra in two variables. 
Cohn proved that the automorphisms of a free Lie algebra with a finite set of generators are tame.
Shestakov and Umirbaev proved that polynomial algebras and free associative algebras in three variables over a field of characteristic zero have wild automorphisms \cite{su04,u07}.
Abdykhalykov, Mikhalev, and Umirbaev
found wild automorphisms of two-generated free Leibniz algebras.
It was proved that free Poisson and free generic Poisson algebras in two variables have no wild automorphisms \cite{ksu18, mltu09}.
These results motivate the following question.

\begin{question}
Are the automorphisms of a two-generated free transposed Poisson algebra tame?
\end{question}

Makar-Limanov and Umirbaev \cite{mlu07} proved an analog of the Bergman Centralizer Theorem: the centralizer of every non-constant element in a free Poisson algebra is a polynomial algebra in one variable. These results motivate the following question.

\begin{question}[Zerui Zhang]
Describe the centralizer of a non-constant element in a free transposed Poisson algebra.
\end{question}

By Corollary \ref{transcontact}, each unital transposed Poisson algebra is an algebra of Jordan brackets.
Martínez,  Shestakov, and   Zelmanov  have shown that all unital Jordan brackets are embeddable into Poisson
brackets \cite{msz01}.
Hence, all unital transposed Poisson brackets are embeddable into Poisson
brackets.
These results motivate the following question.

\begin{question}
Are transposed Poisson brackets embeddable into Poisson brackets?
\end{question}

Thanks to the famous relationship between Poisson algebras and Jordan superalgebras given by the Kantor double, the classification of simple Jordan superalgebras gives a classification of simple Poisson algebras.
Simple generalized Poisson algebras and simple generalized Gerstenhaber algebras (odd generalized Poisson superalgebras) have also been described \cite{ck07,ck10}.
These results motivate the following question.

\begin{question}\footnote{Fernández Ouaridi proved that there are no non-trivial complex finite-dimensional simple transposed Poisson algebras \cite{fer23}.}
\unskip\footnote{Fernández Ouaridi classified simple finite-dimensional transposed Poisson algebras over algebraically closed fields of characteristic $p>3$ \cite{fer26}.}
Classify simple transposed Poisson algebras. \label{q20}
\end{question}

The study of some generalizations of Poisson algebras (noncommutative Poisson algebras;
Leibniz-Poisson algebras \cite{casas};
generic Poisson algebras \cite{ksu18}, etc.) motivates the study of some generalizations of transposed Poisson algebras. 
\begin{question}
Define and study the following generalizations of transposed Poisson algebras:
\begin{enumerate}[(a)]
    \item noncommutative transposed Poisson algebras;
    \item transposed Poisson-Leibniz (in particular, symmetric Leibniz) algebras;
    \item transposed generic Poisson algebras.

\end{enumerate}

\end{question}

One of the important problems in the study of Poisson and transposed Poisson algebras is the description of (transposed) Poisson structures on a given Lie algebra or a given associative algebra.
The most useful tool for describing Poisson structures on a given associative algebra is the study of its derivations and biderivations.
Similarly,
transposed Poisson structures on a given Lie algebra can be described by studying its $\frac{1}{2}$-derivations and $\frac{1}{2}$-biderivations.
A constructive method for describing Poisson structures on a given Lie algebra
was given in a study of Leibniz bialgebras and symmetric Leibniz algebras \cite{said2}.
These results and Theorem \ref{nonass} motivate the following question.

\begin{question}
Give a constructive method for describing transposed Poisson structures
on a given non-unital associative (not necessarily commutative) algebra.
\end{question}

It is known that the Poisson bracket admits an extension to the field of fractions  \cite{kaledin,mlu16,makar12,ksu18}.
By Theorem \ref{tpfield},
a unital transposed Poisson bracket on an algebra without zero divisors admits an extension to the field of fractions.
These results motivate the following question.

\begin{question}\label{q23}\footnote{Kaygorodov and Zhang proved that the free one-generated (non-unital) transposed Poisson algebra embeds into a transposed Poisson field generated by its image, while the free transposed Poisson algebras with at least two generators have zero divisors.}
Does the bracket in Theorem \ref{tpfield} give an extension of a transposed Poisson algebra to the field of fractions in the non-unital case?
\end{question}

The universal multiplicative
enveloping algebras of free Poisson and free generic Poisson fields were studied in \cite{mlu16,ksu18};
in particular, it was proved that they are free ideal rings. These results motivate the following question.

\begin{question}\label{q24}\footnote{Kaygorodov and Zhang gave a positive answer: the universal multiplicative enveloping algebra of every transposed Poisson field $K$ is isomorphic to the Ore extension $K[t;-\mathfrak D]$; in particular, it is a left and right principal ideal domain, and hence a free ideal ring.}
Is the universal multiplicative enveloping algebra of a free transposed Poisson field a free ideal ring?
\end{question}

Multiplicative enveloping algebras are often much simpler than the algebras themselves, and they can be very helpful when the variety is complicated. A basis of the universal enveloping algebra ${\rm P}^e$ of a free Poisson algebra ${\rm P}$ was constructed, and it was
proved that the left dependence of finitely many elements
of ${\rm P}^e$ over ${\rm P}^e$ is algorithmically recognizable \cite{u12}.
These results motivate the following question.

\begin{question}[Vladimir Dotsenko]
Construct a basis of the universal enveloping algebra ${\rm TP}^e$ of a free transposed Poisson algebra ${\rm TP}$.
Is the left dependence of finitely many elements of ${\rm TP}^e$ over ${\rm TP}^e$ algorithmically recognizable?
 \end{question}

It is known that each Novikov--Poisson algebra gives a transposed Poisson algebra with respect to the commutator of the Novikov multiplication \cite{bai20}.
We say that a transposed Poisson algebra is special if it can be embedded into a Novikov--Poisson algebra with respect to the commutator bracket.
Similarly, we say that a transposed Poisson algebra is $D$-special (from ``differentially'') if it can be embedded into a commutative associative algebra with a derivation $\mathfrak D$ with respect to the bracket $[x,y]=\mathfrak D(x)y-x\mathfrak D(y).$ Obviously, every $D$-special transposed Poisson algebra is special. It is also known that the class of special Jordan algebras (i.e., Jordan algebras embeddable into associative algebras with respect to the multiplication $x \circ y =xy+yx$) is a quasivariety but not a variety \cite{sverchkov}. On the other hand, the class of all special Gelfand--Dorfman algebras (i.e., those embeddable into Poisson algebras with a derivation $d$ with respect to the multiplication $x \circ y =xd(y)$) is closed under homomorphic images and thus forms a variety \cite{ks21}.
These results motivate the following question.

\begin{question}[Pavel Kolesnikov]
Does the class of special (resp., $D$-special) transposed Poisson algebras form a variety or a quasivariety?
\end{question}

We call special identities all identities that hold in all special transposed Poisson algebras but do not hold in all transposed Poisson algebras. 
The study of special and non-special identities is a popular topic in non-associative algebra (see, for example, 
 the Jordan algebra case in \cite{gl},
 the Gelfand-Dorfman algebra case in  \cite{kso19},
 the case of dialgebras in \cite{kv13},
 the Jordan trialgebra and post-Jordan algebra cases in \cite{bbm},
etc.).
These results motivate the following question.

\begin{question}[Pavel Kolesnikov]
\label{q27}\footnote{Kaygorodov and Zhang gave a partial answer: they found identities that hold in all special (and hence in all $D$-special) transposed Poisson algebras but do not hold in all transposed Poisson algebras, for example, $\sum_{\sigma\in\mathbb S_4}(-1)^{\sigma}[x_{\sigma(1)},[x_{\sigma(2)},[x_{\sigma(3)},[x_{\sigma(4)},x_0]]]]=0$. In particular, free transposed Poisson algebras with at least two generators are not special.}
Find the identities separating the variety generated by special (resp., $D$-special) transposed Poisson algebras in the variety of all transposed Poisson algebras.
Find the identities separating the variety generated by $D$-special transposed Poisson algebras in the variety generated by special transposed Poisson algebras.
\end{question}

It is also known that each two-generated Jordan algebra is special \cite{shirshov},
each one-generated Jordan dialgebra is special \cite{vv12}, and
each $2$-dimensional Gelfand--Dorfman algebra is special \cite{ks21}.
These results motivate the following question.

\begin{question}\label{q28}\footnote{Kaygorodov and Zhang gave a positive answer: every one-generated transposed Poisson algebra is special. Moreover, the free one-generated transposed Poisson algebra is $D$-special, while its quotients by the ideals generated by $x^n$ $(n\geq3)$, where $x$ is the generator, are not $D$-special.}
Are one-generated transposed Poisson algebras special?
\end{question}

By \cite{mcc92},
the Kantor doubles of all unital transposed Poisson algebras are special Jordan superalgebras.
On the other hand, there are Poisson algebras whose Kantor doubles are not special \cite{mcc92}.
These results motivate the following question.

\begin{question}
Is the Kantor double of a transposed Poisson algebra a special Jordan superalgebra?
\end{question}

An important class of Gelfand–Dorfman algebras is formed by special ones,
i.e., those that can be embedded into differential Poisson algebras (see \cite{kso19}). Namely, a Gelfand–Dorfman algebra $V$ with
operations $[\cdot, \cdot]$ and $*$ is special if there exist a Poisson algebra $(P, \cdot, \{\cdot,\cdot\})$ and a derivation $d$ of $P$ such that $V \subseteq P,$ $[u, v] = \{u, v\},$ and $u * v = d(u)v$ for all $u, v \in V.$
Kolesnikov and Nesterenko proved that
each transposed Poisson algebra obtained from a Novikov--Poisson algebra is a special Gelfand–Dorfman algebra \cite{kn23}.
Sartayev proved that transposed Poisson algebras satisfy the known special identities of
Gelfand–Dorfman algebras \cite{kms}.
These results motivate the following question.

\begin{question}[Pavel Kolesnikov, Bauyrzhan Sartayev]
\label{q30}\footnote{Zhang, Hao, Sun, and Chen gave a positive answer: every transposed Poisson algebra is a special Gelfand--Dorfman algebra \cite{zhsc26a}.}
Are transposed Poisson algebras special Gelfand–Dorfman algebras?
\end{question}

From the geometric point of view, the variety of $n$-dimensional algebras defined by a family of polynomial identities is a subvariety of $\mathbb{C}^{n^3}.$
The geometric classification of a variety of algebras is a description of its geometric variety: its irreducible components, their dimensions, and rigid algebras
(see \cite{ak21,fkkv22,kv20} and references therein).
The variety of $n$-dimensional transposed Poisson algebras
is a subvariety of $\mathbb{C}^{2n^3}.$
The geometric classification of complex $3$-dimensional transposed Poisson algebras was given in \cite{bk22}.
Neretin \cite{ner} gave upper bounds for the dimensions of the varieties of $n$-dimensional associative, commutative, and Lie algebras.
These results motivate the following question.

\begin{question}
Give an upper bound for the dimension of the variety of $n$-dimensional transposed Poisson algebras.
\end{question}

The Koszul property of a (quadratic) operad is quite important; its definition requires more notions
from operad theory \cite{gk94,lv13}. If an operad ${\rm P}$ is Koszul, one can define the corresponding homotopy algebras, ${\rm P}_{\infty}$-algebras. For example,
the operads of Lie, associative, pre-Lie, and Poisson algebras are Koszul, and hence we have ${\rm L}_{\infty}$-, ${\rm A}_{\infty}$-, ${\rm PL}_{\infty}$-, and ${\rm P}_{\infty}$-algebras.
These remarks motivate the following question.

\begin{question}[Li Guo]\footnote{Here is a short proof of the fact that the operad $\mathrm{TP}$ of
transposed Poisson algebras is not Koszul (communicated by Vladimir
Dotsenko). Using Gr\"obner bases for operads
(\url{http://irma.math.unistra.fr/~dotsenko/Operads.html}), one finds
that the dimensions of the components of the operad $\mathrm{TP}$ up to
arity five are $1,2,6,20,74$. Moreover, this operad is known
to be self-dual under the Koszul duality for operads. Therefore, if it
were Koszul, the exponential generating function $f(x)$ of the dimensions
of its components would satisfy the functional equation $f(-f(-x))=x$.
However, the dimensions listed above give
$f(-f(-x))=x + 7/30x^5 + O(x^6)$; therefore, the operad
$\mathrm{TP}$ is not Koszul.}
\footnote{Another proof was given by Askar Dzhumadildaev in \cite{dzhuma}:
he defined weak Leibniz algebras as algebras satisfying
$(ab-ba)c=2a(bc)-2b(ac)$ and
$a(bc-cb)=2(ab)c-2(ac)b,$ and proved that the weak Leibniz operad is not Koszul and that the operads of weak Leibniz algebras and transposed Poisson algebras are isomorphic.}
Is the operad of transposed Poisson algebras Koszul?
If so, there might be a ${\rm TP}_{\infty}$-algebra structure. \end{question}

Some works on the cohomology theory of Poisson algebras \cite{cohom1,cohom2} motivate the following question.

\begin{question}\label{q33}\footnote{Zhang, Hao, Sun, and Chen developed a cohomology theory of transposed Poisson algebras with coefficients in arbitrary modules and the corresponding deformation theory \cite{zhsc26b}.}
Describe the cohomology theory of transposed Poisson algebras.
\end{question}

Poisson algebras are semi-classical limits of quantization
deformations of commutative associative algebras into associative
algebras. $F$-manifold algebras are semi-classical limits of
quantization deformations of commutative associative algebras (or
commutative pre-Lie algebras) into pre-Lie algebras.
These results motivate the following question.

\begin{question}[Chengming Bai]
\label{q34}\footnote{Chen and Bai introduced Novikov deformations of commutative associative algebras and proved that their semi-classical limits are transposed Poisson algebras; in particular, all transposed Poisson algebras of Novikov--Poisson type (including all unital ones) can be quantized in this way \cite{cb24}.}
Into which algebras should commutative algebras be deformed so that transposed Poisson algebras arise as the semi-classical limits of these quantization deformations?
\end{question}

The study of commutative and noncommutative Poisson bialgebras \cite{lbs20,nb13} motivates the following question.

\begin{question}[Chengming Bai]
\footnote{Liu and Bai answered this question by developing a bialgebra theory for transposed Poisson algebras \cite{lb23}.}
Define and study transposed Poisson bialgebras.
\end{question}

The study of double  Poisson algebras \cite{v08} motivates the following question.

\begin{question}\label{q36}\footnote{Double transposed Poisson algebras were introduced and studied by Fairon and Safonkin \cite{fs26}.}
Define and study double transposed Poisson algebras.
\end{question}

Bai, Bai, Guo, and Wu proved that each transposed Poisson algebra
$({\mathfrak L}, \cdot, [\cdot, \cdot])$ with a derivation $\mathfrak D$ of both operations gives a
transposed Poisson $3$-Lie algebra with the multiplication
\[
[x, y, z] = \mathfrak D(x) \cdot  [y, z] - \mathfrak D(y)\cdot [x,z] + \mathfrak D(z)\cdot [x, y]. \]
These results motivate the following question.

\begin{question}[Chengming Bai, Ruipu  Bai, Li Guo \& Yong Wu]
\footnote{Huang, Chen, Chen, and Ding proved this conjecture under a certain strong condition \cite{conj}.}
\unskip\footnote{Mashurov proved that the strong condition used by Huang, Chen, Chen, and Ding fails for the free transposed Poisson $3$-Lie algebra \cite{mash26}.}
Let $n \geq 2$ be an integer. \label{q37} Let $({\mathfrak L}, \cdot, [\cdot, \ldots, \cdot])$ be a transposed Poisson $n$-Lie algebra, let $\mathfrak D$ be a derivation of $({\mathfrak L}, \cdot)$ and $({\mathfrak L}, [\cdot, \ldots, \cdot])$. Define an $(n+1)$-ary operation
\[\llbracket x_1,  \ldots, x_{n+1} \rrbracket := \sum\limits_{i=1}^{n+1} (-1)^{i+1}{\mathfrak D}(x_i) \cdot [x_1, \ldots, \hat{x}_i, \ldots,  x_{n+1}], \]
where $\hat{x}_i$ means that the $i$-th entry is omitted. 
Is $({\mathfrak L}, \cdot, \llbracket \cdot,  \ldots, \cdot \rrbracket)$  a transposed Poisson
$(n + 1)$-Lie algebra?
\end{question}

Cantarini and Kac classified all complex linearly compact simple
 (generalized) Poisson $n$-Lie algebras \cite{ck16}, which motivates the following question.

\begin{question}\label{q38}\footnote{Mashurov proved that a transposed Poisson $n$-Lie algebra is simple if and only if its $n$-Lie part is simple, and gave an example of a simple linearly compact transposed Poisson $n$-Lie algebra \cite{mash26}.}
Give examples and classify simple transposed Poisson $n$-Lie algebras.
\end{question}

It is known that the variety of Gerstenhaber algebras coincides with the variety of odd Poisson superalgebras (see, for example,  \cite{ck10}).
This remark motivates the following question.

\begin{question}
Define and study transposed Gerstenhaber algebras.
\end{question}

It is known \cite{bai20} that the tensor product of transposed Poisson algebras is equipped with a canonical transposed Poisson algebra structure. 
On the level of operads, this means that the operad $\mathrm{TPA}$ of transposed Poisson algebras is a Hopf operad. Many natural examples of Hopf operads arise in algebraic topology
when one computes the homology of a topological operad. By analogy with transposed Poisson algebras, one may define transposed Gerstenhaber
algebras, which are given by similar axioms but with a shifted Lie bracket of homological degree one (and appropriate signs arising from the Koszul sign rule). The corresponding
operad seems to be a Hopf operad as well. This suggests the following question. 

\begin{question}[Vladimir Dotsenko]
Is it the homology operad of some operad made of topological spaces? In the case of the Gerstenhaber operad, the corresponding topological operad
is the operad of little 2-disks, one of the protagonists of operad theory in its early days. 
\end{question}

\appendix
\section*{Appendix A. Differences from the published version}
\addcontentsline{toc}{section}{Appendix A. Differences from the published version}

\noindent\textbf{Subsection ``Transposed Poisson structures and quasi-automorphisms''} (Definition~10, Proposition~11, and Lemma~12 of the published version): the subsection was erroneous and meaningless and has been removed.  

\noindent\textbf{Proposition~\ref{homctr} and Theorem~\ref{tpfield}:} the statement has been changed.

\noindent\textbf{Lemma~\ref{kacmoody}:} the statement is false; it has been replaced by the opposite one (affine Kac--Moody algebras admit non-trivial transposed Poisson structures).

\noindent\textbf{Section~\ref{open}:}   notes on the answers to Questions~\ref{q1}, \ref{q15}, \ref{q16}, \ref{q20}, \ref{q23}, \ref{q24}, \ref{q27}, \ref{q28}, \ref{q30}, \ref{q33}, \ref{q34}, \ref{q36}, \ref{q37}, and~\ref{q38} have been added.

\end{document}